\documentclass[amssymb,amstex,amsmath,10pt,a4paper]{amsart}
\usepackage{amsmath}
\usepackage{amssymb}
\usepackage{latexsym,amsfonts,url}
\usepackage{epsfig}

\theoremstyle{plain}
\newtheorem{thm}{Theorem}[section]
\newtheorem{prop}[thm]{Proposition}
\newtheorem{cor}[thm]{Corollary}
\newtheorem{lem}[thm]{Lemma}
\newtheorem{rem}[thm]{Remark}
\newtheorem{p}[thm]{Problem}

\theoremstyle{definition}
\newtheorem{df}[thm]{Definition}
\newtheorem{ex}[thm]{Example}

\theoremstyle{remark}
\newtheorem*{Rmk}{Remark}

\numberwithin{equation}{section}

\newcommand{\vol}{{\rm Vol}}
\newcommand{\dist}{{\rm dist}}
\newcommand{\area}{{\rm Area}}
\newcommand{\length}{{\rm Length}}
\newcommand{\cv}{{\rm Conv}}
\newcommand{\cone}{\mbox{$\times\hspace*{-0.228cm}\times$}}
\newcommand{\md}{{\rm{d}}}

\newcommand{\hr}{{\mathbb{H}^2 \times\mathbb{R}}}
\newcommand{\tr}{{\rm{tr}}}
\newcommand{\tc}{{\rm TotalCurvature}}

\begin{document}
\title[Embeddedness of minimal submanifolds]{Embeddedness of proper minimal submanifolds\\in homogeneous spaces}
\author[Sung-Hong Min]{Sung-Hong Min}

\begin{abstract}
We prove the three embeddedness results as follows. $({\rm i})$ Let $\Gamma_{2m+1}$ be a piecewise geodesic Jordan curve with $2m+1$ vertices in
$\mathbb{R}^n$, where $m$ is an integer $\geq2$. Then the total curvature of $\Gamma_{2m+1}<2m\pi$. In particular, the total curvature of
$\Gamma_5<4\pi$ and thus any minimal surface $\Sigma \subset \mathbb{R}^n$ bounded by $\Gamma_5$ is embedded. Let $\Gamma_5$ be a piecewise geodesic
Jordan curve with $5$ vertices in $\mathbb{H}^n$. Then any minimal surface $\Sigma \subset \mathbb{H}^n$ bounded by $\Gamma_5$ is embedded. If
$\Gamma_5$ is in a geodesic ball of radius $\frac{\pi}{4}$ in $\mathbb{S}^n_+$, then $\Sigma \subset \mathbb{S}^n_+$ is also embedded. As a
consequence, $\Gamma_5$ is an unknot in $\mathbb{R}^3$, $\mathbb{H}^3$ and $\mathbb{S}^3_+$. $({\rm ii})$ Let $\Sigma$ be an $m$-dimensional proper
minimal submanifold in $\mathbb{H}^n$ with the ideal boundary $\partial_{\infty} \Sigma = \Gamma$ in the infinite sphere
$\mathbb{S}^{n-1}=\partial_\infty \mathbb{H}^n$. If the M{\"o}bius volume of $\Gamma$ $\widetilde{\vol}(\Gamma) < 2\vol(\mathbb{S}^{m-1})$,
then $\Sigma$ is embedded. If $\widetilde{\vol}(\Gamma) = 2\vol(\mathbb{S}^{m-1})$, then $\Sigma$ is embedded unless it is a cone. $({\rm iii})$ Let $\Sigma$ be a proper minimal surface in $\hr$. If $\Sigma$ is vertically regular at infinity and has two ends, then $\Sigma$ is embedded.\\

\end{abstract}

\maketitle

%%%%%%%%%%%%%%%%%%%%%%%%%%%%%%%%%%%%%%%%%%%%%%%%%%%%%%%%%%%%%%%%%%%%%%
\section{Introduction}
%%%%%%%%%%%%%%%%%%%%%%%%%%%%%%%%%%%%%%%%%%%%%%%%%%%%%%%%%%%%%%%%%%%%%%
To decide whether a minimal submanifold is embedded or not is one of the significant problems in minimal surface theory. The first well-known
embeddedness theorem is given by Rad{\' o} \cite{Rado2}. He proved that if a Jordan curve $\Gamma$ in $\mathbb{R}^n$ has a 1-1 projection onto the
boundary of a convex domain $D$ in a plane then any minimal surface bounded by $\Gamma$ is a graph over $D$. In the 1970's Tomi and Tromba \cite{TT}
and Almgren and Simon \cite{AS} showed that an extremal Jordan curve $\Gamma$ in $\mathbb{R}^3$ spans an embedded minimal surface. A curve is
extremal if it lies in the boundary of a convex domain. Moreover Meeks and Yau \cite{MY} proved that the Douglas-Morrey solution of the Plateau
problem is embedded under the same condition.

In $1929$, Fenchel \cite{F1} proved that the total curvature of any closed curve in $\mathbb{R}^3$ is always greater than or equal to $2\pi$ and is
equal to $2\pi$ if and only if it is a convex curve in a plane. He observed that the total curvature of a regular curve $\Gamma$ is measured by the
length of spherical image of the unit tangent vectors to $\Gamma$. F{\'a}ry \cite{Fary} and Milnor \cite{M} proved independently that the total
curvature of a knot in $\mathbb{R}^3$ is greater than $4\pi$. For minimal surfaces, it had been open for a long time whether a minimal surface
bounded by a Jordan curve with total curvature at most $4\pi$ is embedded or not.

In $2002$, Ekholm, White, and Wienholtz \cite{EWW} proved the embeddedness of any minimal surface bounded by a Jordan curve $\Gamma$ in $\mathbb{R}^n$ with total curvature at most $4\pi$. This gives a simple proof of F{\'a}ry-Milnor theorem. Choe and Gulliver \cite{CG2} generalized this result for minimal surfaces in an $n$-dimensional complete simply connected Riemannian manifold with sectional curvature bounded above by a non-positive constant and for minimal surfaces in $\mathbb{S}_+^n$. In particular, they proved that any minimal surface bounded by a Jordan curve $\Gamma$ in $\mathbb{H}^n$ ($\mathbb{S}^n_+$, resp.) with total curvature less than or equal to $4\pi+\inf_{p\in\Sigma} \area(p\cone\Gamma)$ ($4\pi-\sup_{p\in\Sigma} \area(p\cone\Gamma)$, resp.) is always embedded unless it is a cone. It follows that a Jordan curve in $\mathbb{H}^3$ ($\mathbb{S}^3_+$, resp.) with total curvature less than or equal to $4\pi+\inf_{p\in\Sigma} \area(p\cone\Gamma)$ ($4\pi-\sup_{p\in\Sigma} \area(p\cone\Gamma)$, resp.) is unknotted.\\

In this paper we will prove three embeddedness results for some proper minimal submanifolds in homogeneous manifolds. In order to obtain embeddedness
of a minimal submanifold $\Sigma$, we are going to get an estimate
\begin{equation*}
\Theta_\Sigma (p) < 2,
\end{equation*}
where $\Theta_\Sigma (p)$ is the density of $\Sigma$ at $p$.\\

In section \ref{sect:compact}, we will first deal with the following two problems.
\begin{p}
{\rm Let $\Gamma_k$ be a piecewise geodesic Jordan curve with $k$ vertices in $\mathbb{R}^n$. What is the upper bound for the total curvature of
$\Gamma_k$?}
\end{p}
If $k=3$, then Fenchel's theorem \cite{F1} implies that the total curvature of $\Gamma_3$ is equal to $2\pi$. If $k \geq 4$ is an even integer, then
it is not difficult to show that the total curvature of $\Gamma_{k}$ is less than $k\pi$ and that $k\pi$ is sharp. We will find the sharp upper bound
for the total curvature of $\Gamma_{k}$, where $k \geq 5$ is an odd integer. In case of this, $k=2m+1$ for some integer $m\geq2$. The theorem we will
prove is as follows:
\begin{displaymath}
\tc(\Gamma_{2m+1})<2m\pi.
\end{displaymath}
In particular,
\begin{displaymath}
\tc(\Gamma_{5})<4\pi,
\end{displaymath}
and thus any minimal surface $\Sigma$ bounded by $\Gamma_5$ is embedded by \cite{EWW}, and $\Gamma_5$ is an unknot in $\mathbb{R}^3$ by \cite{Fary}
and \cite{M}. This leads us to Problem \ref{prob2} in a very natural way.
\begin{p} \label{prob2}
{\rm Let $\Gamma_5$ be a piecewise geodesic Jordan curve with $5$ vertices in $\mathbb{H}^n$ or $\mathbb{S}^n_+$. If $\Sigma$ is a minimal surface
bounded by $\Gamma_5$, is $\Sigma$ embedded? Is $\Gamma_5$ an unknot in $\mathbb{H}^3$ or $\mathbb{S}^3_+$?}
\end{p}
We will prove the following. Any minimal surface $\Sigma$ bounded by $\Gamma_5$ is embedded in $\mathbb{H}^n$. In case of hemisphere
$\mathbb{S}^n_+$, if $\Gamma_5$ is a piecewise geodesic Jordan curve lying in a geodesic ball of radius $\frac{\pi}{4}$, then $\Sigma \subset
\mathbb{S}^n_+$ is also embedded. As a consequence, $\Gamma_5$ is an unknot in $\mathbb{H}^3$ and $\mathbb{S}^3_+$.

\begin{figure}[h]
\begin{center}
\includegraphics[height=5cm, width=5cm]{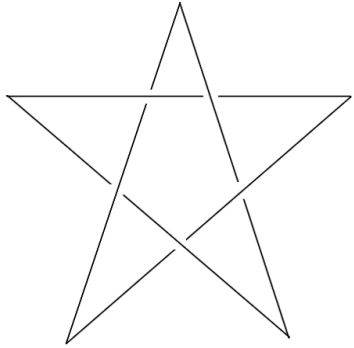}
\end{center}
\caption{Non-existence of a star-shaped knotted piecewise geodesic Jordan curve with $5$ vertices in $\mathbb{R}^3$, $\mathbb{H}^3$ and $\mathbb{S}^3_+$. (This picture cannot exist!)}
\end{figure}

Unlike \cite{CG2}, there is no need to assume that the total curvature of $\Gamma_5 \subset \mathbb{H}^n$ or $\mathbb{S}^n_+$ is bounded above by a
constant and additional terms. In fact, the total curvature and the additional term including $\area(p\cone\Gamma)$ are very difficult to compute in
$\mathbb{H}^n$ or $\mathbb{S}^n_+$. In this regard, we can see that these results are intuitive and efficient when we consider piecewise geodesic
Jordan curves.

In particular, there does not exist a star-shaped knotted piecewise geodesic Jordan curve with $5$ vertices in $\mathbb{R}^n$, $\mathbb{H}^n$ and $\mathbb{S}^n_+$ (Figure 1). Here, $n=5$ is the largest number for a piecewise geodesic Jordan curve $\Gamma_n$ with $n$ vertices to be unknotted.\\

While these results deal with the embeddedness of a compact surface, it was Schoen \cite{S} who first dealt with the embeddedness of a complete minimal surface. He proved that a complete minimal hypersurface in $\mathbb{R}^n$ which has two ends and is regular at infinity is embedded and that it is actually the catenoid. He also proved that if a minimal surface in $\mathbb{R}^3$ is regular at infinity then it has finite total curvature and thus it is proper \cite{Fang}. Levitt and Rosenberg \cite{LR} showed that a connected minimal hypersurface in $\mathbb{H}^n$ which has ideal boundary $\Gamma =S_1 \cup S_2$, where $S_1,S_2$ are disjoint round spheres in $\mathbb{S}^{n-1}=\partial_\infty \mathbb{H}^n$, and is regular at infinity is the catenoid. Here, $\Sigma$ being regular at infinity implies $\overline{\Sigma}=\Sigma \cup \Gamma$.\\

In section \ref{sect:hyperbolic}, we will prove the embeddedness of any $m$-dimensional proper minimal submanifold in $\mathbb{H}^n$ which has an
ideal boundary $\Gamma$ in the infinite sphere $\mathbb{S}^{n-1} = \partial_\infty \mathbb{H}^n$ whose M{\"{o}}bius volume is at most
$2\vol(\mathbb{S}^{m-1})$.  The M{\"{o}}bius volume $\widetilde{\vol}(\Gamma)$ of $\Gamma$ is defined to be
\begin{displaymath}
\widetilde{\vol}(\Gamma)=\sup\{{\vol}_\mathbb{R} (g(\Gamma)) \,|\,
{g \in {\text{\rm M{\"{o}}b}}(\mathbb{S}^{n-1})}\},
\end{displaymath}
where ${\vol}_\mathbb{R} (g(\Gamma))$ denotes the volume of
$g(\Gamma)$ in $\mathbb{S}^{n-1}$.

Consider a $2$-dimensional cone $0 \cone \Gamma$ in $\mathbb{R}^n$. The total geodesic curvature of $\Gamma$ is less than or equal to the length of the radial projection of $\Gamma$ onto the unit sphere, and this length is equal to $2\pi$ times the density of $0 \cone \Gamma$ at $0$, as shown in \cite{EWW} and \cite{CG2}. In view of this, we think of the M{\"{o}}bius volume or the volume in a sphere as a kind of total curvature and in this way we can obtain a suitable estimate of the density of a minimal submanifold.\\

In section \ref{sect:hr}, we will deal with minimal surfaces in $\hr$. $\hr$ is a homogeneous $3$-dimensional Riemannian manifold and is one of Thurston's eight geometry. In 2002, Rosenberg \cite{Ro} constructed infinitely many disk type minimal surfaces in $\hr$ which are graphs over ideal polygons in $\mathbb{H}^2$ by generalizing the Jenkins-Serrin type theorem to $\hr$. After that, many mathematicians have been working on the theory of minimal surfaces in $\hr$. In particular, they have constructed many minimal surfaces in $\hr$, such as the catenoid with the vertical axis of revolution, ruled minimal surfaces, the Riemann type minimal surface which is foliated by horizontal curves of constant curvature, and vertical minimal graphs over an unbounded domain in $\mathbb{H}^2$ which are not necessarily convex and not necessarily bounded by convex arcs (\cite{NR}, \cite{Haus}, \cite{ST}). Recently Pyo \cite{Pyo} constructed complete annular minimal surfaces with finite total curvature.

The last theorem we will prove is as follows: A proper minimal surface in $\hr$ which is vertically regular at infinity and has two ends is embedded.\\

We would like to mention that the problems of this paper were proposed by Jaigyoung Choe.

%%%%%%%%%%%%%%%%%%%%%%%%%%%%%%%%%%%%%%%%%%%%%%%%%%%%%%%%%%%%%%%%%%%%%%%%%%%%%%%%%%%%
\section{On minimal surfaces bounded by $\Gamma_5$} \label{sect:compact}
%%%%%%%%%%%%%%%%%%%%%%%%%%%%%%%%%%%%%%%%%%%%%%%%%%%%%%%%%%%%%%%%%%%%%%%%%%%%%%%%%%%%
Let $k$ be a positive integer $\geq 2$. A \textit{geodesic polygonal curve} (we will just use the term \textit{polygon} in $\mathbb{R}^n$) is a piecewise length-minimizing geodesic curve $\Gamma_k$ with $k$ vertices $v_0(=v_k), v_1, \cdots, v_{k-1}$ such that $v_{i+1} \neq v_i$ for $i=0,\cdots, {k-1}$. Denote $\Gamma_k$ by $v_0 v_1 \cdots v_{k-1}$. In particular, if $\Gamma_k$ is a closed geodesic polygonal curve, i.e. $v_k=v_0$, then we will denote it as $v_0 v_1 \cdots v_{k-1} v_0$.

%%%%%%%%%%%%%%%%%%%%%%%%%%%%%%%%%%%%%%%%%%%%%%%%%%%%%%%%%%%%%%%%%%%%%%%%%%%%%%%%%%%%
\subsection{Total curvature of $\Gamma_k$ in $\mathbb{R}^n$}
%%%%%%%%%%%%%%%%%%%%%%%%%%%%%%%%%%%%%%%%%%%%%%%%%%%%%%%%%%%%%%%%%%%%%%%%%%%%%%%%%%%%

Let $\Gamma_k$ be a closed polygon in $\mathbb{R}^n$ with $k$ vertices. Fenchel's theorem \cite{F1} implies that the total curvature of $\Gamma_3$ is equal to $2\pi$. If $\Gamma_k$ is simple, i.e. it is a Jordan curve, then
\begin{displaymath}
\tc(\Gamma_k) < k\pi.
\end{displaymath}
Here, $k\pi$ is sharp if $k$ is even integer. Consider a closed polygon $\widetilde{\Gamma} = v_0 v_1 \cdots v_{k-1} v_0$ where $v_{\text{even}}=v_0$ and $v_{\text{odd}}=v_1$ ($\widetilde{\Gamma}$ is a line segment as a set). We can make it a polygonal Jordan curve $\Gamma$ arbitrary close to $\widetilde{\Gamma}$ by moving vertices of $\widetilde{\Gamma}$ slightly. This implies that there is a polygonal Jordan curve $\Gamma$ of the total curvature $k\pi-\epsilon$ for small $\epsilon>0$.

The main goal of this section is to find the sharp upper bound for the total curvature of $\Gamma_k$, where $k$ is an odd integer $\geq 5$, in $\mathbb{R}^n$.

In the following lemma, we will call a closed geodesic polygonal curve with $3$ vertices a \textit{geodesic triangle}.

\begin{lem} \label{lem:gamma5}
Let $\Delta$ be a geodesic triangle $p_0 p_1 p_2 p_0$ in $\mathbb{S}^{n-1}$. Then
\begin{equation*}
\length(\Delta) \leq 2\pi.
\end{equation*}
The equality holds if and only if either $\{p_0,p_1,p_2\}$ contains antipodal points or $\Delta=p_0 p_1 p_2 p_0$ is a great circle.
\end{lem}

\begin{proof}
Suppose that $\{p_0,p_1,p_2\}$ contains antipodal points. We may assume that $p_0$ and $p_1$ are antipodal. Then $p_1 p_2 p_0$ is a geodesic of length $\pi$ no matter where $p_2$ lies in. Therefore
\begin{displaymath}
\length(\Delta) = \length(p_0 p_1) + \length(p_1 p_2 p_0) = 2\pi.
\end{displaymath}
On the other hand, if any two points in $\{p_0,p_1,p_2\}$ are not antipodal, then each geodesic segment of $\Delta$ has the length $< \pi$. Observe that $p_0,p_1,p_2$ lie in a hemisphere. There are two possibilities, either $p_0,p_1,p_2$ lie in an $(n-1)$-dimensional open hemisphere or completely in $\mathbb{S}^{n-2}$. If $p_0,p_1,p_2 \in \mathbb{S}^{n-2}$, then either $p_0,p_1,p_2$ lie in an $(n-2)$-dimensional open hemisphere or in $\mathbb{S}^{n-3}$ for the same reason mentioned above. Even in the worst case, in a finite step we can conclude that $p_0,p_1,p_2$ lie in the same open hemisphere unless $p_0, p_1, p_2$ are contained in a great circle $S$.

Suppose that $p_0,p_1,p_2$ lie in the same $k$-dimensional open hemisphere. Let $O$ be the origin and $O\cone \Delta$ be the cone over $\Delta$ with the vertex $O$ in $\mathbb{R}^n$. Note that the length of $p_i p_{i+1}$ is the same as the angle between line segments $O p_i$ and $O p_{i+1}$, $i=0,1,2$. We can develop the tetrahedron $O-p_0 p_1 p_2 \subset \mathbb{R}^n$ into a plane since $O-p_0 p_1 p_2$ is flat. But to make the tetrahedron $O-p_0 p_1 p_2$ from a development drawing described in a plane, the angle around $O$ must be strictly less than $2\pi$. Therefore $\length(\Delta) < 2\pi$.

If $p_0, p_1, p_2$ are contained in a great circle $S$, then it is easy to show that $\length(\Delta) \leq \length(S) = 2\pi$. Moreover, $\length(\Delta) = \length(S)$ if and only if $\Delta=p_0 p_1 p_2 p_0$ is a great circle $S$ itself.
\end{proof}

\begin{prop} \label{prop:gamma5-1}
Let $p_0$ and $p_2$ be points in $\mathbb{S}^{n-1}$. Let $\Lambda$ be a geodesic polygonal curve $p_0 p_1 p_2$ in $\mathbb{S}^{n-1}$ which is made by adding one point $p_1$. If $\length(p_0 p_2)=\theta$, then
\begin{equation*}
\length(\Lambda) \leq 2\pi-\theta.
\end{equation*}
The equality holds if and only if either $\{p_0,p_1,p_2\}$ contains antipodal points or $p_0 p_1 p_2 p_0$ is a great circle.
\end{prop}

\begin{proof}
Let $\Delta$ be the geodesic triangle $p_0 p_1 p_2 p_0$ in $\mathbb{S}^{n-1}$. Then
\begin{equation*}
\length(p_0 p_1 p_2 p_0) = \length(\Lambda) + \length(p_2 p_0).
\end{equation*}
The conclusion follows Lemma \ref{lem:gamma5}.
\end{proof}

\begin{prop} \label{prop:gamma5-2}
Let $p_0$ and $p_3$ be points in $\mathbb{S}^{n-1}$. Let $\Lambda$ be a geodesic polygonal curve $p_0 p_1 p_2 p_3$ in $\mathbb{S}^{n-1}$ which is made by adding two points $p_1$ and $p_2$. If $\length(p_0 p_3)=\theta$, then
\begin{equation*}
\length(\Lambda) \leq 2\pi+\theta.
\end{equation*}
If we further assume that each geodesic segment of $\Lambda$ has the length $< \pi$ and $\theta< \pi$, then the equality holds if and only if $p_0 p_3 p_1 p_2 p_0$ is a great circle and $\length(p_2 p_0 p_3) <\pi$.

If $\theta=\pi$ and the equality holds, then $p_1=p_3$ is the antipodal point of $p_2=p_0$.

In general (without assumption mentioned above), $p_0, p_1, p_2, p_3$ lie in the same great circle if equality holds.
\end{prop}

\begin{proof}
Let $\alpha$ be the length of $p_2 p_3$ and $\beta$ be the distance between $p_0$ and $p_2$. Then we have
\begin{eqnarray}
\length(\Lambda) &=& \length(p_0 p_1 p_2)+\length (p_2 p_3) \nonumber\\
&\leq& (2\pi-\beta) + \alpha \label{eqn:prop2}\\
&\leq& 2\pi +\theta. \nonumber
\end{eqnarray}
Here, (\ref{eqn:prop2}) comes from Proposition \ref{prop:gamma5-1}. The second inequality holds because $p_2 p_3$ is a length-minimizing geodesic.

Suppose that the equality holds. Then $\alpha=\beta+\theta$.

One of the following holds.
\begin{itemize}
\item [(\textit{i})] If $\alpha=\pi$, then $p_2$ and $p_3$ are antipodal;
\item [(\textit{ii})] If $\alpha<\pi$, then $\length(p_2 p_3)=\length(p_2 p_0 p_3)<\pi.$
\end{itemize}
If we further assume that each geodesic segment of $\Lambda$ has the length $< \pi$ and $\theta< \pi$, then we have $0<\beta<\pi$. Because if $\beta=\pi$ then $\theta=0$ and $\length(p_2 p_3)=\alpha=\pi$, and if $\beta=0$ then $\length(p_0 p_1)$ is equal to $\pi$ from the equality condition of (\ref{eqn:prop2}). That is a contradiction. Thus there is the unique great circle $S$ determined by $p_0$ and $p_2$. Since $p_1 \in S$ by (\ref{eqn:prop2}), $p_0 p_1 p_2 p_0 = S$. Together with ($ii$), this implies that $p_0 p_3 p_1 p_2 p_0=S$ and $\length(p_2 p_0 p_3) <\pi$. The converse can be shown directly.

In the case $\theta=\pi$, $\length(\Lambda)=3\pi$. Therefore each geodesic segment of $\Lambda$ has the length $\pi$.

Lastly, we will prove that $p_0, p_1, p_2, p_3$ lie in the same great circle when the equality holds. There is a trichotomy.
 \begin{itemize}
\item If $0<\beta<\pi$, then we already know that there is the unique great circle $S$ determined by $p_0$ and $p_2$ and that $p_1 \in S$. We only need to consider the case (\textit{i}). But in this case $p_3$ is the antipodal point of $p_2$.
\item If $\beta=\pi$, then $p_0=p_3$ and $p_2$ is the antipodal point of $p_0$.
\item If $\beta=0$, then $p_0=p_2$ and $p_1$ is the antipodal point of $p_0$.
\end{itemize}
For all of these three cases, we can conclude that $p_0, p_1, p_2, p_3$ lie in the same great circle.
\end{proof}

Before stating the general case, we are going to deal with $\Gamma_5 \subset \mathbb{R}^n$.
\begin{thm} \label{thm:gamma5}
Let $\Gamma_5$ be a piecewise geodesic Jordan curve with $5$ vertices in $\mathbb{R}^n$. Then
\begin{equation*}
\tc(\Gamma_5) < 4\pi.
\end{equation*}
\end{thm}

\begin{proof}
Let $v_0(=v_5), v_1, v_2, v_3, v_4$ be the vertices of $\Gamma_5$ and let $T$ be the \textit{tangent indicatrix of $\Gamma_5$}, that is the spherical image of the unit tangent vectors to $\Gamma_5$. Then $T$ is a closed geodesic polygonal curve (in general, not simple) in $\mathbb{S}^{n-1}$ with $5$ vertices $p_i$, where $p_i$ is the unit vector parallel to $v_i v_{i+1}$, for $i=0,1,2,3,4$. It is known that the length of $p_i p_{i+1}$ in $\mathbb{S}^{n-1}$ is the same as the turning angle around $v_{i+1}$ in $\mathbb{R}^n$ (\cite{M}). Since $\Gamma_5$ can not have a cuspidal point, the length of a geodesic segment of $T$ is strictly less than $\pi$ and thus $T$ is length-minimizing.

One can make a closed geodesic polygonal curve $T$ in $\mathbb{S}^{n-1}$ with $5$ vertices as follows: For given two points $p_0$ and $p_2$ in $\mathbb{S}^{n-1}$, $T$ can be identified with $p_0 p_1 p_2 p_3 p_4 p_0$ by adding three points $p_1$, $p_3$, $p_4$.

Let the distance between $p_0$ and $p_2$ be $\theta$. Joining Proposition \ref{prop:gamma5-1} and Proposition \ref{prop:gamma5-2}, we have the following inequality:
\begin{eqnarray*}
\tc(\Gamma_5)&=&\length(T)\\ &=& \length(p_0 p_1 p_2) +\length(p_0 p_4 p_3 p_2)\\
&\leq&(2\pi-\theta)+(2\pi+\theta)=4\pi.
\end{eqnarray*}
By a contradiction argument, we will prove that the inequality is strict. Suppose that the equality holds. If $\theta=\pi$, then $p_0=p_3$, $p_2=p_4$, and moreover $p_2$ and $p_3$ are antipodal points from Proposition \ref{prop:gamma5-2}. It is a contradiction. Hence $\theta<\pi$. Then the equality condition implies that $p_0 p_1 p_2 p_0$ and $p_0 p_2 p_4 p_3 p_0$ are great circles, respectively. Actually they coincide. Denote this great circle by $S$. Since $p_0 p_1 p_2 p_0$ and $p_0 p_2 p_4 p_3 p_0$ have an opposite direction, the tangent indicatrix $T=p_0 p_1 p_2 p_3 p_4 p_0$ of $\Gamma_5$ winds $S$ twice. Therefore $\Gamma_5$ is a planar curve and of rotation index $2$. It implies that $\Gamma_5$ has self-intersection. But it is a contradiction.
\end{proof}

The next proposition is a bridge to the general theorem about the total curvature for a piecewise geodesic Jordan curve $\Gamma_k$, where $k$ is an odd integer $\geq 5$. For a convenience, write $k=2m+1$ for $m\geq2$.
\begin{prop} \label{prop:gamma5 in Sn}
{\rm I(m)}: Let $\Gamma_{2m+1}$ be a closed geodesic polygonal curve in $\mathbb{S}^{n-1}$ with ${2m+1}$ vertices. Then
\begin{equation} \label{eqn:4pi in Sn}
\length(\Gamma_{2m+1}) \leq 2m\pi.
\end{equation}
If the equality holds, then every vertex of $\Gamma_{2m+1}$ lies in the same great circle $S$. Moreover if every geodesic segment of $\Gamma_{2m+1}$ has the length $< \pi$, then $\Gamma_{2m+1}$ winds $S$ $m$-times.

{\rm J(m)}: Let $p_0$ and $p_{2m}$ be points in $\mathbb{S}^{n-1}$. Let $\Lambda_{2m+1}$ be a geodesic polygonal curve $p_0 p_1 \cdots p_{2m-1} p_{2m}$ which is made by adding $2m-1$ points $p_1, \cdots, p_{2m-1}$. If $\length(p_0 p_{2m})=\theta$, then
\begin{equation*}
\length(\Lambda_{2m+1}) \leq 2m\pi-\theta.
\end{equation*}
If the equality holds, then every vertex of $\Lambda_{2m+1}$ lies in the same great circle $S$. Moreover if $\theta<\pi$ and every geodesic segment of $\Lambda_{2m+1}$ has the length $< \pi$, then $p_0 p_1 \cdots p_{2m} p_0$ winds $S$ $m$-times.
\end{prop}

\begin{proof}
Use the mathematical induction in $m\geq 2$ as follows:
\begin{itemize}
\item We already proved I(2) on the way to prove Theorem \ref{thm:gamma5}.
\item If I(m) holds then so does J(m) for $m \geq 2$.
\item For a given two points $p_0$ and $p_3$ in $\mathbb{S}^{n-1}$, a closed geodesic polygonal curve $\Gamma_{2m+3}$ is made by adding $2m+1$ points as follows: Add two points $p_1,p_2$ and $2m-1$ points $p_4,\cdots, p_{2m+2}$, respectively. And then identify $\Gamma_{2m+3}$ with $p_0 p_1 \cdots p_{2m+2}p_0$. Joining J(m) and Proposition \ref{prop:gamma5-2}, we have I(m+1) in the same manner as Theorem \ref{thm:gamma5}.
\end{itemize}
\end{proof}

\begin{thm} \label{thm:gammak}
Let $m$ be an integer $\geq 2$. Let $\Gamma_{2m+1}$ be a piecewise geodesic Jordan curve with $2m+1$ vertices in $\mathbb{R}^n$. Then
\begin{equation} \label{eqn:oddpi}
\tc(\Gamma_{2m+1}) < 2m\pi.
\end{equation}
\end{thm}

\begin{proof}
The tangent indicatrix $T$ of $\Gamma_{2m+1}$ is a closed geodesic polygonal curve in $\mathbb{S}^{n-1}$ with $2m+1$ vertices. By (\ref{eqn:4pi in Sn}), we have
\begin{displaymath}
\tc(\Gamma_{2m+1})=\length(T)\leq 2m\pi.
\end{displaymath}
Since $\Gamma_{2m+1}$ is a Jordan curve, every geodesic segment of $T$ has the length $< \pi$. If the equality holds, then $T$ is a great circle $S$ as a set and winds itself $m$-times. It is a contradiction.
\end{proof}

In (\ref{eqn:oddpi}), the upper bound $2m\pi$ is sharp. Let $\widetilde{\Gamma}=v_0 v_1 \cdots v_{2m-1} v_0$ be a closed polygon in $\mathbb{R}^n$ with $2m$ vertices, where $v_{\text{even}}=v_0$ and $v_{\text{odd}}=v_1$. Take new point $v_{2m}$, which is different from $v_0$ and $v_{2m-1}$, in a geodesic segment of $\widetilde{\Gamma}$. We can move vertices $v_0,\cdots, v_{2m-1},v_{2m}$ slightly to make $\widetilde{\Gamma}$ a polygonal Jordan curve $\Gamma$ with $2m+1$ vertices. It is possible to construct such a $\Gamma$ arbitrary close to $\widetilde{\Gamma}$. Therefore there is a polygonal Jordan curve $\Gamma$ of the total curvature $2m\pi-\epsilon$ for small $\epsilon>0$.

In particular, $\Gamma_5 \subset \mathbb{R}^n$ has interesting properties as follows:
\begin{cor} \label{cor:unknot}
Let $\Gamma_5$ be a piecewise geodesic Jordan curve with $5$ vertices in $\mathbb{R}^n$. Then any minimal surface $\Sigma$ in $\mathbb{R}^n$ bounded by $\Gamma_5$ is embedded. If $\Gamma_5 \subset \mathbb{R}^3$, then it is an unknot.
\end{cor}
\begin{proof}
See \cite{EWW} for embeddedness, and see F{\'a}ry-Milnor theorem (\cite{Fary}, \cite{M}) for unknottedness. (For reference, the latter one can be also obtained in a different way using Theorem 4.8 in \cite{M}).
\end{proof}

Corollary \ref{cor:unknot} leads us to the following problem in a very natural way.
\begin{p} \label{problem}
Let $\Gamma_5$ be a piecewise geodesic Jordan curve with $5$ vertices in $\mathbb{H}^n$ or $\mathbb{S}^n_+$. If $\Sigma$ is a minimal surface bounded by $\Gamma_5$, is $\Sigma$ embedded? Is $\Gamma_5$ an unknot in $\mathbb{H}^3$ or $\mathbb{S}^3_+$?
\end{p}

In the next section, we will give an answer to Problem \ref{problem}.

%%%%%%%%%%%%%%%%%%%%%%%%%%%%%%%%%%%%%%%%%%%%%%%%%%%%%%%%%%%%%%%%%%%%%%%%%%
\subsection{On minimal surfaces bounded by $\Gamma_5$ in $\mathbb{H}^n$ and $\mathbb{S}^n_+$}
%%%%%%%%%%%%%%%%%%%%%%%%%%%%%%%%%%%%%%%%%%%%%%%%%%%%%%%%%%%%%%%%%%%%%%%%%%

Let $N$ be an $n$-dimensional Riemannian manifold. The injectivity radius $i(N)$ of $N$ is the largest $r$ such that the exponential map is an embedding on an open ball of radius $r$ in $T_p N$ for all $p$. Observe that $i(\mathbb{R}^n)=i(\mathbb{H}^n)=\infty$ and $i(\mathbb{S}^n)=\pi$.

\begin{df} {(\cite{CG1})}
Let $N$ be an $n$-dimensional space form. Let $\Gamma \subset N$ be a $k$-dimensional rectifiable set in $N$ and let $p$ be a point in $N$ such that $\dist(p,q)<i(N)$ for all $q \in \Gamma$. Let $J(\rho)$ be a radial function on $N$ as follows: $J(\rho)=\sin \rho$ in $\mathbb{S}^n$, $\rho$ in $\mathbb{R}^n$ and $\sinh \rho$ in $\mathbb{H}^n$. The \textit{$k$-dimensional angle $A^k(\Gamma,p)$ of $\Gamma$ viewed from $p$} is defined by setting
\begin{equation*}
A^k(\Gamma,p)=\frac{\vol((p \cone \Gamma) \cap S_\rho(p))}{J(\rho)^k},
\end{equation*}
where $S_\rho(p)$ is the geodesic sphere of radius $\rho<\dist(p,\Gamma)$ centered at $p$, and the volume is measured counting multiplicity. Clearly the angle does not depend on $\rho$.
\end{df}

Note that
\begin{equation} \label{eqn:angle2}
A^k(\Gamma,p)=(k+1)\omega_{k+1}\Theta_{p\cone \Gamma}(p).
\end{equation}

Before stating the theorem, we will sketch models of $\mathbb{S}^n$ in a similar way to Section \ref{sect:hyperbolic}. Note that $\mathbb{S}^n$ is isometrically immersed onto the unit sphere $S_1$, $x_0^2+x_1^2+\cdots+x_n^2=1$ in $\mathbb{R}^{n+1}$ endowed with the usual metric, $\md s^2=\md x_0^2+\md x_1^2+\cdots+\md x_n^2$. Let $\rho$ be the distance in $\mathbb{S}^n$ measured from the south pole, $(-1,0,\cdots,0)$.
Then $|x|=\sqrt{x_1^2+\cdots+x_n^2}=\sin \rho$ and $\md |x|=\cos \rho \cdot \md \rho$.
We can consider $\mathbb{S}^n$ as $\mathbb{R}^n \cup \{\infty\}$ via the stereographic projection $\Phi:S_1 \subset \mathbb{R}^{n+1} \rightarrow \mathbb{R}^n \cup \{\infty\}$ of $S_1$ onto $\mathbb{R}^n \cup \{\infty\}$ which is given by
\begin{displaymath}
\Phi\left(x_0,x_1,\cdots,x_n\right)=\left(u_1,\cdots,u_n \right) \hspace{3mm} \text{where } u_i=\frac{x_i}{1-x_0},\, i=1,\cdots,n.
\end{displaymath}
Since $\Phi$ is 1-1 and onto, there is the inverse $\Phi^{-1}: \mathbb{R}^n \cup \{\infty\} \rightarrow
S_1$ such that $(u_1,\cdots,u_n) \mapsto (x_0,x_1,\cdots,x_n)$ where $x_0=\frac{r^2-1}{r^2+1}$,
$x_i=\frac{2u_i}{1+r^2}$, $i=1,\cdots,n$.
Then $\Phi^{-1}$ induces on $\mathbb{R}^n \cup \{\infty\}$ the metric $\md s_{\rm B}^2=\frac{4\md
s_\mathbb{R}^2}{(1+r^2)^2}$ where $\md s_\mathbb{R}^2=\md u_1^2+\cdots+\md u_n^2$ is the Euclidean metric and $r^2=u_1^2+\cdots+u_n^2$.
Note that
\begin{displaymath}
\sin \rho=\frac{2r}{1+r^2},\, \rho=\arcsin \frac{2r}{1+r^2} \hspace{3mm} \text{and } r=\cot \frac{1}{2} \rho.
\end{displaymath}

To estimate the density, we need the following proposition.
\begin{prop} {\rm (\cite{CG2}) (Density comparison)} \label{prop:density comparison}\\
\indent {\rm (1)} Let $\Gamma$ be a (piecewise) $C^2$ immersed closed curve in $\mathbb{S}^n_+$. $p \in \mathbb{S}^n_+$ such that $\dist(p,\Gamma)\leq\frac{\pi}{2}$. Let $\Sigma$ be a branched minimal surface in $\mathbb{S}^n_+$ with boundary $\Gamma$. Then
\begin{equation} \label{eqn:angle3}
\Theta_\Sigma (p) < \Theta_{p\cone \Gamma}(p),
\end{equation}
unless $\Sigma$ is a totally geodesic.

{\rm (2)} Let $\Gamma$ be a (piecewise) $C^2$ immersed closed curve in $\mathbb{H}^n$. Let $\Sigma$ be a branched minimal surface in $\mathbb{H}^n$ with boundary $\Gamma$ in $\mathbb{H}^n$. Then
\begin{equation*}
\Theta_{\Sigma}(p)<\Theta_{p\cone \Gamma} (p),
\end{equation*}
unless $\Sigma$ is a totally geodesic.
\end{prop}

\begin{proof}
See \cite{CG2}. The original proof in \cite{CG2} only deals with a regular $C^2$ curve $\Gamma$. However the density comparison (which is shown by using monotonicity) at the vertex of $\Gamma$ can be obtained in the same manner even though $\Gamma$ is a piecewise $C^2$ curve.
\end{proof}

The original version of the second part of Proposition \ref{prop:density comparison} is more powerful. It is proved for $\Sigma$ which is a branched minimal surface in an $n$-dimensional simply connected Riemannian manifold with sectional curvature $\leq -\kappa^2$.

The following is the main theorem of Section \ref{sect:compact}.

\begin{thm} \label{thm:gamma5 in Sn}
Let $\Gamma_5$ be a piecewise geodesic Jordan curve with $5$ vertices in a geodesic ball $B_r \subset \mathbb{S}^n_+$ of radius $r < \frac{\pi}{4}$. Then any minimal surface $\Sigma \subset B_r$ bounded by $\Gamma_5$ is embedded.
\end{thm}

\begin{proof}
Let $p \in \Sigma \subset B_r \subset \mathbb{S}^n_+$. Since $B_r$ is convex, $p \in \cv(\Gamma_5)$ where $\cv(\Gamma_5)$ is the convex hull of $\Gamma_5$. Thus $\dist(p, \Gamma_5) < \frac{\pi}{2}$. We may assume that $p$ is identified with the origin in $\mathbb{R}^n \cup \{\infty\}$ via the stereographic projection. Joining (\ref{eqn:angle3}) with (\ref{eqn:angle2}) we get
\begin{displaymath}
(k+1)\omega_{k+1} \Theta_\Sigma (p) \leq A^k(\Gamma,p).
\end{displaymath}
Therefore
\begin{eqnarray*}
\Theta_\Sigma (p) &\leq& \frac{\length_\mathbb{S}((p \cone \Gamma_5) \cap S_\rho(p))}{2\pi \sin \rho}\\
&=&\frac{1}{2\pi} \int_{(p \cone \Gamma_5) \cap S_\rho(p)} \frac{1}{\sin \rho}\md \sigma^{\mathbb{S}}\\
&=&\frac{1}{2\pi} \int_{\Phi(p \cone \Gamma_5) \cap S_{\rho^*}(0)} \frac{1}{r} \cdot{\left(\frac{1+r^2}{2}\right)} \md \sigma^{\mathbb{S}}\\
&=&\frac{1}{2\pi} \int_{\Phi(p \cone \Gamma_5) \cap S_{\rho^*}(0)} \frac{1}{r} \md \sigma^{\mathbb{R}}\\
&=&\frac{1}{2\pi} \length_\mathbb{R}(\Phi(p \cone \Gamma_5) \cap S_1(0)),
\end{eqnarray*}
where $\rho^*=\cot \frac{1}{2}\rho$.

There are three cases. (From now on, we will consider $\mathbb{S}^n_+$ as a subset of $\mathbb{R}^n \cup \{\infty\}$ through the stereographic projection $\Phi$. For a convenience, we will omit $\Phi$ if it is not ambiguous.)

Case I. $p \in \Sigma \setminus \Gamma_5$.

Observe that $(p \cone \Gamma_5) \cap S_1(0)$ is a piecewise length-minimizing geodesic closed curve with $5$ vertices in the unit sphere in $\mathbb{R}^n$. Therefore Proposition \ref{prop:gamma5 in Sn} implies $\length_\mathbb{R}((p \cone \Gamma_5) \cap S_1(0)) \leq 4\pi$. Now we claim that the equality can not occur. If the equality holds, then all of the vertices of $(p \cone \Gamma_5) \cap S_1(0)$ lie in the same great circle $S$. This implies that all of the vertices of $\Phi(\Gamma_5)$ are in the plane $P$ containing the origin and $S$. Since each geodesic segment (this is not a line segment, in general) of $\Phi(\Gamma_5)$ minimizes length, $\Gamma_5 \subset P$ and thus $(p \cone \Gamma_5) \cap S_1(0) \subset P$ also. This implies that $(p \cone \Gamma_5) \cap S_1(0)$ is a geodesic circle $S$ as a set and winds $S$ twice. But it is a contradiction since $\Phi(\Gamma_5)$ is a Jordan curve. Therefore we have
\begin{displaymath}
\Theta_\Sigma (p) < 2.
\end{displaymath}

Case II. $p \in \Gamma_5 \setminus \{\text{vertices}\}$.

$T_p\Gamma_5$ intersects $S_1(0)$ at two points $a,b$ which are antipodal. And thus $(p \cone \Gamma_5) \cap S_1(0)$ is a piecewise length-minimizing geodesic curve with $5$ vertices and $a,b$ are end points of $(p \cone \Gamma_5) \cap S_1(0)$. Proposition \ref{prop:gamma5 in Sn} implies that $\length_\mathbb{R}((p \cone \Gamma_5) \cap S_1(0)) \leq 3\pi$. In the same manner to the first case, we can show that the equality can not occur. Therefore we have
\begin{displaymath}
\Theta_\Sigma (p) < \frac{3}{2}.
\end{displaymath}

Case III. $p$ is a vertex of $\Gamma_5$.

Let $\theta$ be an exterior angle of $\Gamma_5$ at $p$. Then Proposition \ref{prop:gamma5-2} implies that $\length_\mathbb{R}((p \cone \Gamma_5) \cap S_1(0)) \leq 2\pi+(\pi-\theta)=3\pi-\theta$. In a similar way to the first two cases, we have
\begin{displaymath}
\Theta_\Sigma (p) < \frac{3}{2}-\frac{\theta}{2\pi}.
\end{displaymath}
These three density estimates complete the proof.
\end{proof}

We can also prove the following in $\mathbb{H}^n$.

\begin{thm} \label{thm:gamma5 in Hn}
Let $\Gamma_5$ be a piecewise geodesic Jordan curve with $5$ vertices in $\mathbb{H}^n$. Then any minimal surface $\Sigma$ bounded by $\Gamma_5$ is embedded.
\end{thm}

\begin{proof}
The proof is similar to Theorem \ref{thm:gamma5 in Sn}.
\end{proof}

\begin{cor} \label{cor:star}
There does not exist a star-shaped knotted piecewise geodesic Jordan curve in $\mathbb{H}^3$ or in a geodesic ball of radius $<\frac{\pi}{4} \subset \mathbb{S}^3_+$ which consists of $5$ geodesic segments {\rm (See Figure 1 in Introduction)}.
\end{cor}

\begin{proof}
Note that a minimal disk in a $3$-dimensional manifold bounded by a knotted Jordan curve always has not only self-intersections but also a branch point. It is a direct consequence of Theorem \ref{thm:gamma5 in Hn} and Theorem \ref{thm:gamma5 in Sn}.
\end{proof}

In Corollary \ref{cor:star}, $n=5$ is the critical number for a piecewise geodesic Jordan curve $\Gamma_n$ with $n$ vertices to be unknotted. It is not difficult to find a knotted piecewise geodesic Jordan curve $\Gamma_6$ with $6$ vertices (Figure 2).

\begin{figure}[h]
\begin{center}
\includegraphics[height=5cm, width=5cm]{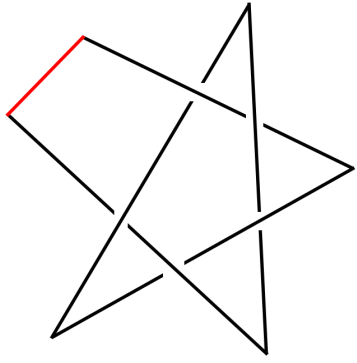}
\end{center}
\caption[Five is critical! Existence of a knotted piecewise geodesic Jordan curve $\Gamma_6$ with $6$ vertices]{Five is critical! Existence of a knotted piecewise geodesic Jordan curve $\Gamma_6$ with $6$ vertices.}
\end{figure}

%%%%%%%%%%%%%%%%%%%%%%%%%%%%%%%%%%%%%%%%%%%%%%%%%%%%%%%%%%%%%%%%%%%%%%%%%%%%%%%%%%%%
\section{Embeddedness of proper minimal submanifolds in $\mathbb{H}^n$} \label{sect:hyperbolic}
%%%%%%%%%%%%%%%%%%%%%%%%%%%%%%%%%%%%%%%%%%%%%%%%%%%%%%%%%%%%%%%%%%%%%%%%%%%%%%%%%%%%
Let us sketch two models of hyperbolic space $\mathbb{H}^n$.
First, $\mathbb{H}^n$ is isometrically immersed onto the hyperboloid ${\rm H}$, $-x_0^2+x_1^2+\cdots+x_n^2=-1$, $x_0>0$, in $\mathbb{R}^{n+1}$ endowed with the Minkowski metric, $\md s_{\mathbb{L}}^2=-\md x_0^2+\md x_1^2+\cdots+\md x_n^2$, which is denoted by $\mathbb{L}^{n+1}$.
Let $\rho$ be the distance in $\mathbb{H}^n$ measured from $(1,0,\cdots,0)$.
Then $|x|=\sqrt{x_1^2+\cdots+x_n^2}=\sinh \rho$ and $\md |x|=\cosh \rho \cdot \md \rho$.
Second, one can consider $\mathbb{H}^n$ as the unit ball ${\rm
B}^n=\{(u_1,\cdots,u_n)|u_1^2+\cdots+u_n^2<1\} \subset
\mathbb{R}^n$.
Define a mapping $\Psi:{\rm H} \subset \mathbb{L}^{n+1} \rightarrow {\rm B}^n \subset \mathbb{R}^n$ by
\begin{displaymath}
\Psi\left(x_0,x_1,\cdots,x_n\right)=\left(u_1,\cdots,u_n \right) \hspace{3mm} \text{where } u_i=\frac{x_i}{1+x_0},\, i=1,\cdots,n.
\end{displaymath}
Actually $\Psi$ is known as the stereographic projection of ${\rm H}$ onto the unit ball in the hyperplane
$\{x_0=0\}$.
Since $\Psi$ is 1-1 and onto, there is the inverse $\Psi^{-1}: {\rm B}^n \rightarrow
{\rm H}$ such that $(u_1,\cdots,u_n) \mapsto (x_0,x_1,\cdots,x_n)$ where $x_0=\frac{1+r^2}{1-r^2}$,
$x_i=\frac{2u_i}{1-r^2}$, $i=1,\cdots,n$.
Then $\Psi^{-1}$ induces on ${\rm
B}^n$ the metric $\md s_{\rm B}^2=\frac{4\md
s_\mathbb{R}^2}{(1-r^2)^2}$ where $\md s_\mathbb{R}^2=\md u_1^2+\cdots+\md u_n^2$ is the Euclidean metric and $r^2=u_1^2+\cdots+u_n^2$.
Such a ball ${\rm
B}^n$ is called the Poincar{\' e} ball, one of the models of $\mathbb{H}^n$.
Note that
\begin{displaymath}
\sinh \rho=\frac{2r}{1-r^2},\, \rho=\log \frac{1+r}{1-r} \hspace{3mm} \text{and } r=\tanh \frac{1}{2} \rho.
\end{displaymath}

If we employ the Poincar{\' e} ball, then the ideal boundary $\partial_{\infty} \Sigma$ of $\Sigma \subset \mathbb{H}^n$ is defined to be the set of all accumulation points of $\Sigma$ in $\mathbb{S}^{n-1}$. Here $\partial_\infty \mathbb{H}^n$, the ideal boundary of $\mathbb{H}^n$, is identified with $\mathbb{S}^{n-1}$.

\begin{df}
Let $\Gamma$ be an $(m-1)$-dimensional submanifold in an $n$-dimensional Riemannian manifold $M$ and let $p$ be a point of $M$. The \textit{$m$-dimensional cone over $\Gamma$ with the vertex $p$} is defined as the union of the geodesic segment from $p$ to $q$, over all $q \in \Gamma$ and is denoted by $p\cone \Gamma$.
\end{df}

From now on, $\vol_\mathbb{R}(\Gamma)$ denotes the volume of $\Gamma$ in $\mathbb{S}^{n-1} \subset \mathbb{R}^n$. In particular $\vol_\mathbb{R}(\mathbb{S}^{m-1})=m\omega_m$, where $\omega_m$ is the volume of the $m$-dimensional unit ball in $\mathbb{R}^m$.

\begin{prop} {\rm(Density estimation)} \label{prop:density estimate}
Let $\Gamma$ be an $(m-1)$-dimensional compact submanifold of $\mathbb{S}^{n-1}$. Let $\Sigma$ be an $m$-dimensional
proper minimal submanifold in $\mathbb{H}^n$ with
$\partial_{\infty} \Sigma = \Gamma \subset
\mathbb{S}^{n-1}=\partial_\infty \mathbb{H}^n$ and let $q$ be a point of $\Sigma$. Let $\psi$ be an isometry of $\mathbb{H}^n$ such that $\psi(q)=(1,0,\cdots,0) \in {\rm H}$ in the hyperboloid model of $\mathbb{H}^n$. Then
\begin{equation} \label{eqn:hyp density}
m \omega_m \Theta_\Sigma(q) \leq \vol_\mathbb{R}(\widetilde{\Gamma}),
\end{equation}
where $\widetilde{\Gamma}$ denotes the ideal boundary of $\psi(\Sigma)$.
If the equality holds, then $\overline{\Sigma}=q\cone \Gamma$.
\end{prop}

The properness of $\Sigma$ in Proposition \ref{prop:density estimate} can be replaced by the hypothesis that $\overline{\Sigma}=\Sigma \cup \Gamma$.

\begin{proof}
Let $G(x)$ be Green's function of $\mathbb{H}^m$, whose derivative is $\sinh^{1-m} x$ for $0<x<\infty$, where $x$ is the distance from a fixed point in $\mathbb{H}^m$. Choe
and Gulliver \cite{CG1} proved that if $\Sigma$ is an
$m$-dimensional minimal submanifold of $\mathbb{H}^n$ and $q \in
\Sigma$ then $G \circ \rho$ is subharmonic on $\Sigma \setminus
\{q\}$ and is harmonic except $q$ if $\Sigma$ is a cone with the vertex $q$. That is,
\begin{equation*}
\triangle_{\Sigma} G(\rho) = m \frac{\cosh \rho}{\sinh^m \rho} \left(1-|\nabla_{\Sigma}
\rho|^2 \right) \geq 0,
\end{equation*}
where $\rho(\cdot)$ is the distance from $q$ in
$\mathbb{H}^n$.

Let $B_r(q)$ denote the geodesic ball in $\mathbb{H}^n$ of radius
$r$ centered at $q$. Integrate $\triangle_{\Sigma} G(\rho)$ over $\Sigma \cap B_R(q)
\setminus B_\epsilon(q)$ for small $\epsilon >0$ and large $R$ and then apply the divergence theorem. Since $\partial({\Sigma \cap B_R(q) \setminus B_\epsilon(q)})=({\Sigma \cap \partial B_\epsilon (q)}) \cup ({\Sigma \cap \partial B_R (q)})$, it implies
\begin{equation*}
0 \leq \int_{\Sigma \cap B_R(q) \setminus B_\epsilon(q)} \triangle_{\Sigma}
G (\rho) = \int_{\Sigma \cap \partial B_\epsilon (q)} \nabla_{\Sigma} G (\rho)\cdot \nu
+ \int_{\Sigma \cap \partial B_R (q)} \nabla_{\Sigma} G (\rho)\cdot \nu,
\end{equation*}
where $\nu$ is the outward unit conormal vector to $\Sigma$.

Recall that $\nabla_{\Sigma} G(\rho)=\sinh^{1-m} \rho \cdot \nabla_{\Sigma} \rho$. Hence
\begin{equation} \label{eqn:hyp est1}
-\int_{\Sigma \cap \partial B_\epsilon (q)} \frac{1}{\sinh^{m-1} \rho} \frac{\partial \rho}{\partial \nu} \leq \int_{\Sigma \cap \partial B_R
(q)} \frac{1}{\sinh^{m-1} \rho} \frac{\partial \rho}{\partial \nu}.
\end{equation}
Along ${\Sigma \cap \partial B_\epsilon (q)}$ $\frac{\partial \rho}{\partial \nu} \rightarrow -1$ uniformly and
\begin{displaymath}
\frac{\vol({\Sigma \cap \partial B_\epsilon (q)})}{\sinh^{m-1}
\epsilon} \rightarrow m \omega_m \Theta_\Sigma (q) \hspace{3mm} \text{as }\epsilon \rightarrow 0.
\end{displaymath}
It follows that
\begin{equation*}
    \lim_{\epsilon \rightarrow 0} \int_{\Sigma \cap \partial B_\epsilon (q)} \frac{1}{\sinh^{m-1} \rho} \frac{\partial \rho}{\partial \nu} = - m \omega_m \Theta_\Sigma (q).
\end{equation*}
Then (\ref{eqn:hyp est1}) yields
\begin{equation} \label{eqn:hyp est2}
m \omega_m \Theta_\Sigma (q) \leq \int_{\Sigma \cap \partial
B_R (q)} \frac{1}{\sinh^{m-1} \rho} \frac{\partial \rho}{\partial \nu}
\,\,\,.
\end{equation}

Let us write $\md \sigma^{\mathbb{H}}$ as
the hyperbolic volume form of $\Sigma \cap
\partial B_R (q)$ in the hyperboloid and Poincar{\' e} ball model in common.
Since $|\frac{\partial \rho}{\partial \nu}| \leq 1$,
\begin{equation*}
\int_{\Sigma \cap \partial B_R (q)} \frac{1}{\sinh^{m-1} \rho} \frac{\partial \rho}{\partial \nu} \leq \int_{\Sigma \cap \partial B_R (q)}
\frac{1}{\sinh^{m-1} \rho}\,\,\,.
\end{equation*}

On the other hand, we can obtain some interesting equalities as follows
\begin{eqnarray} \label{eqn:hyp est3}
\int_{\Sigma \cap \partial B_R (q)} \frac{1}{\sinh^{m-1} \rho}
\,\md \sigma^{\rm \mathbb{H}} &=& \int_{\psi(\Sigma) \cap \partial B_R (\psi(q))}
\left(\frac{1}{\sqrt{x_0^2-1}}\right)^{m-1}\md \sigma^{\rm \mathbb{H}}\\
&=& \int_{\psi(\Sigma) \cap \partial \widetilde{B}_{R^*} (0)} \frac{1}{r^{m-1}}
{\left( \frac{1-r^2}{2} \right)}^{m-1} \md \sigma^{\rm
\mathbb{H}}\nonumber\\ &=& \int_{\psi(\Sigma) \cap \partial \widetilde{B}_{R^*} (0)}
\frac{1}{r^{m-1}} \,\md \sigma^{\mathbb{R}},\nonumber
\end{eqnarray}
where $R^*={\tanh \frac{1}{2} R}$, $\widetilde{B}_{R^*} (0)$ is a geodesic ball in ${\rm B}^n \subset \mathbb{R}^n$ centered at the origin and $\md \sigma^{\mathbb{R}}$ is the volume form of ${\psi(\Sigma) \cap \partial \widetilde{B}_{R^*} (0)}$ in $\mathbb{R}^n$. The last equality holds because $\md s_\mathbb{H}^2$ is conformal to $\md s_\mathbb{R}^2$.

Note that the last integral equals just $\vol({\psi(\Sigma) \cap \partial \widetilde{B}_{R^*} (0)})$ divided by ${R^*}^{m-1}$. In fact it is the volume of the radial projection of $\psi(\Sigma) \cap
\partial \widetilde{B}_{R^*} (0)$ in $\mathbb{R}^n$ onto $\mathbb{S}^{n-1}$. And it converges to $\vol_\mathbb{R}(\partial_\infty \psi(\Sigma))$ as $R^* \rightarrow 1$, that is, as $R \rightarrow \infty$. Hence from (\ref{eqn:hyp est2}) and
(\ref{eqn:hyp est3}) we have (\ref{eqn:hyp density}).

If equality holds, then $\triangle_{\Sigma} G(\rho)$ vanishes on the whole $\Sigma$ with respect to the fixed point $q \in \Sigma$. It implies that $|\nabla_{\Sigma} \rho| \equiv 1$ on $\Sigma$. Let $s \in \Sigma$ and let $\gamma$ be a geodesic such that $\gamma(0)=q$ and $\gamma(1)=s$. Then $\gamma'(1) \in T_s\Sigma$ for all $s \in \Sigma$ because $\nabla_{\Sigma} \rho \in T_s\Sigma$. It then follows that $\Sigma$ is a cone with the vertex $q$.
\end{proof}

\begin{df}
Let $\Gamma$ be an $(m-1)$-dimensional compact submanifold of
$\mathbb{S}^{n-1}$. Let
{M{\"{o}}b}($\mathbb{S}^{n-1}$) be the group of all {M{\"{o}}bius}
transformations of $\mathbb{S}^{n-1}$. The {\it M{\"{o}}bius volume} of $\Gamma$
is defined to be
\begin{equation*}
\widetilde{\vol}(\Gamma)=\sup\{{\vol}_\mathbb{R} (g \circ \Gamma)
\,|\, {g \in {\text{\rm M{\"{o}}b}}(\mathbb{S}^{n-1})}\}.
\end{equation*}
\end{df}

\begin{Rmk}
According to the definition of Li and Yau \cite{LY}, the M{\"{o}}bius volume of $\Gamma$ is the same as the $(n-1)$-conformal
volume of the inclusion of $\Gamma$ into
$\mathbb{S}^{n-1}$.
\end{Rmk}

\begin{prop} \label{prop:example}
Let $\Gamma$ be an $(m-1)$-dimensional compact submanifold of $\mathbb{S}^{n-1}$. Then $\widetilde{\vol}(\Gamma) \geq m\omega_m$. And equality holds if $\Gamma$ is an $(m-1)$-dimensional sphere. In particular, if $\Gamma$ is a closed curve in $\mathbb{S}^2$, then $\widetilde{\vol}(\Gamma) = 2\pi$ if and only if $\Gamma$ is a circle.
\end{prop}

\begin{proof}
Let $p$ be a point of $\Gamma$. There is $\varphi_\epsilon \in \text{M{\"{o}}b}(\mathbb{S}^{n-1})$ fixing $p$ and corresponding to the homothety $\widetilde{\varphi_\epsilon}$ in $\mathbb{R}^{n-1}$ which is defined as $\widetilde{\varphi_\epsilon}(x):=\frac{x}{\epsilon}$. Then $\varphi_\epsilon(\Gamma)$ converges to an $(m-1)$-dimensional great sphere as $\epsilon$ goes to $0$. Therefore $\widetilde{\vol}(\Gamma) \geq m\omega_m$.

Let $\Gamma$ be an $(m-1)$-dimensional sphere. Since the {M{\"{o}}b}ius transformation of $\mathbb{S}^{n-1}$ maps the spheres to the spheres, obviously we have $\widetilde{\vol}(\Gamma)=m\omega_m$.

Let $\Gamma$ be a closed curve in $\mathbb{S}^2$. We only need to
prove that if $\widetilde{{\rm Vol}}(\Gamma)=2\pi$ then $\Gamma$
is a circle. If $\Gamma$ has a self-intersection then we can take
a closed embedded subarc $\gamma$ from $\Gamma$ and clearly
$\widetilde{{\rm Vol}}(\gamma)\leq\widetilde{{\rm Vol}}(\Gamma)$.
Thus it is enough to consider an embedded $\Gamma$ of length
$\leq 2 \pi$. Then $\Gamma$ lies in a closed hemisphere by Horn's
theorem \cite{H}.

Suppose $\Gamma$ is not a circle. Let $D_1$, $D_2 \subset \mathbb{S}^2$
be the domains bounded by $\Gamma$. Then there is a largest circle
$S_i$ in Closure$(D_i)$, $i=1,2$, such that $\Gamma \cap S_i$
consists of at least two points. Choose $\varphi \in \text{
M{\"{o}}b}(\mathbb{S}^2)$ in such a way that $\varphi(S_1)$ and
$\varphi(S_2)$ become two parallels of equal
latitude in northern and southern hemisphere, respectively. Let
$A \subset \mathbb{S}^2$ be the annulus between $\varphi(S_1)$ and
$\varphi(S_2)$.

Now we claim that no closed hemisphere in $\mathbb{S}^2$ can
contain $\varphi(\Gamma)$. Suppose, on the contrary, that
$\varphi(\Gamma)$ lies in a closed hemisphere $U$. Since
$\varphi(\Gamma)$ is not null-homotopic in $A$, $U\cap A$ cannot
be simply connected, and so $\partial U$ lies in $A$ and is not
null-homotopic in $A$. Moreover, assuming that
$\varphi(S_1)\subset U$, we have $\partial U\cap
\varphi(S_2)\neq \emptyset$ since $\Gamma\cap S_i\neq \emptyset$
for $i=1,2$. However, we should note that $\partial U$ intersects
$\varphi(S_i)$ only at one point, for $i=1,2$. But this
contradicts the hypothesis that $\Gamma\cap S_i$ consists of at
least two points.

Therefore no closed hemisphere in $\mathbb{S}^2$ can contain
$\varphi(\Gamma)$ and hence it follows from \cite{H} that
Vol$(\varphi(\Gamma))>2\pi$. This is a contradiction to our
hypothesis $\widetilde{{\rm Vol}}(\Gamma)=2\pi$, and thus we can
conclude that $\Gamma$ is a circle.
\end{proof}

The following is a non-trivial example of a Jordan curve having $\widetilde{\vol}(\Gamma)<4\pi$.

\begin{ex}
Let $S$ and $S^\perp$ be great circles in $\mathbb{S}^2$ and let $p_1$ and $p_2$ be the intersection points of $S$ and $S^\perp$. We can choose four points $p_{ij}$ different from $p_1$ and $p_2$ as follows: $p_{ij} \in S$, $\dist(p_i, p_{ij})=\epsilon<\frac{\pi}{2}$ and $p_{1j}$ and $p_{2j}$ are antipodal for $i,j=1,2$. Then we have new piecewise smooth Jordan curve $\Gamma$ from $S \cup S^{\perp}$ removing length-minimizing geodesic segments connecting $p_{i1}$ and $p_{i2}$ and adding semi-circles $S_j$ of length $\pi$ with the end points $p_{1j}$ and $p_{2j}$ which intersects $S$ at a right angle for $i,j=1,2$

Let $\varphi$ be any {M{\"{o}}b}ius transformation of $\mathbb{S}^2$. Then $S$ and $S_j$ remain still part of circles under $\varphi$ and intersection angle between $\varphi(S)$ and $\varphi(S_j)$ is $\frac{\pi}{2}$ from the conformality , $j=1,2$. It is not difficult to show that
\begin{displaymath}
\length(\varphi(S_1))+\length(\varphi(S_2)) \leq 2\pi
\end{displaymath}
and thus $\length(\varphi(\Gamma))<4\pi$. It follows that $\widetilde{\vol}(\Gamma)<4\pi$.
\end{ex}

\begin{thm} \label{thm:hyperb}
Let $\Gamma$ be an $(m-1)$-dimensional compact submanifold of
$\mathbb{S}^{n-1}$. Let $\Sigma$ be an $m$-dimensional proper
minimal submanifold in $\mathbb{H}^n$ with $\partial_{\infty}
\Sigma = \Gamma$. If $\widetilde{\vol}(\Gamma) < 2m \omega_m$,
then $\Sigma$ is embedded. If $\widetilde{\vol}(\Gamma) = 2m \omega_m$, then $\Sigma$ is embedded unless it is a cone.
\end{thm}

\begin{proof}
Let $p$ be a point on $\Sigma$ in $\mathbb{H}^n$. In accordance with {Proposition \ref{prop:density estimate}}
\begin{equation} \label{eqn:sph1}
m \omega_m \Theta_\Sigma(p) \leq
\vol_\mathbb{R}(\widetilde{\Gamma}).
\end{equation}
Let Isom$(\mathbb{H}^n)$ be the group of all isometries of $\mathbb{H}^n$. Let {M{\"{o}}b}($\rm{B}^{n})$ be the group of all {M{\"{o}}b}ius transformations of $\rm{B}^{n}$. Then (Chapter 4 in \cite{R}),
\begin{equation*}
    \text{Isom}(\mathbb{H}^n) \simeq \text{M{\"{o}}b}(\rm{B}^{n}) \simeq \text{M{\"{o}}b}(\mathbb{S}^{n-1}).
\end{equation*}
Given $\psi \in \text{Isom}(\mathbb{H}^n)$, we may consider it as in $\text{M{\"{o}}b}(\mathbb{S}^{n-1})$. Then it follows that
\begin{equation}\label{eqn:sph2}
\vol_\mathbb{R}(\widetilde{\Gamma}) \leq \widetilde{\vol}(\Gamma) < 2m\omega_m.
\end{equation}
Therefore combining (\ref{eqn:sph1}) and (\ref{eqn:sph2}), we have
\begin{equation*}
\Theta_\Sigma(p) < 2,
\end{equation*}
and hence $\Sigma$ is embedded.

If $\widetilde{\vol}(\Gamma)=2m \omega_m$, then $\Theta_\Sigma(p) \leq 2$ for every $p \in \Sigma$. Let $q \in \Sigma$ be a point of density 2. Since equality holds in (\ref{eqn:hyp density}), then it is a cone with the vertex $q$. This completes the proof.
\end{proof}

%%%%%%%%%%%%%%%%%%%%%%%%%%%%%%%%%%%%%%%%%%%%%%%%%%%%%%%%%%%%%%%%%%%%%%%%%%%%%%%%%%%%
\section{Embeddedness of proper minimal surfaces in $\hr$} {\label{sect:hr}}
%%%%%%%%%%%%%%%%%%%%%%%%%%%%%%%%%%%%%%%%%%%%%%%%%%%%%%%%%%%%%%%%%%%%%%%%%%%%%%%%%%%%
\begin{prop} \label{prop:hr subh}
Let $\Sigma$ be a complete minimal surface in $\hr$ and $p \in \Sigma$. Let $\rho$ be the distance from $p$ in $\hr$. Then
\begin{equation*}
\triangle_{\Sigma} \log \rho \geq 0
\end{equation*}
on $\Sigma \setminus \{p\}$.
\end{prop}

Before proving {Proposition \ref{prop:hr subh}}, we will determine the Jacobi fields along a unit speed geodesic in $\hr$. In this section we use the Poincar{\' e} disk model of $\mathbb{H}^2$,
\begin{equation*}
\mathbb{H}^2=\{(u_1,u_2) \in \mathbb{R}^2 | r^2=u_1^2+u_2^2 <1\}.
\end{equation*}
As the product space, $\hr$ has the coordinates $(u_1,u_2,z)$ endowed with the metric
\begin{equation*}
\md \tilde{s}^2 = \frac{4(\md u_1^2 + \md u_2^2)}{(1-r^2)^2} + \md z ^2,
\end{equation*}
where $(u_1,u_2) \in \mathbb{H}^2$ and $z \in \mathbb{R}$.
Let $p$ be a point in $\hr$ and $\gamma$ be a unit speed geodesic in $\hr$ emanating from $p$ with $\gamma(0)=p$ and $\gamma'(0)=v \in T_p \hr$. Since $\hr$ is a homogeneous space, there exists the isometry $\varphi$ of $\hr$ so that $\varphi(p)=0$ and $\md \varphi(v)=(c,0,\sqrt{1-4c^2})=:w$ for some $c \in [0,\frac{1}{2}]$. For convenience denote by $\gamma$ the geodesic $\varphi \circ \gamma$, i.e. $\gamma(0)=0$ and $\gamma'(0)=w$.

To find the geodesic $\gamma$ explicitly, recall the geodesic equation as follows:
\begin{equation} \label{eqn:geod eqn}
\gamma_k''(t)+\sum_{i,j=1}^3 \Gamma_{ij}^k \gamma_i'(t) \gamma_j'(t) =0, \hspace{1cm} k=1,2,3.
\end{equation}
Since the Christoffel symbols of the Riemannian connection is given by
\begin{equation*}
\Gamma_{ij}^k = \sum_{l=1}^3 \frac{1}{2} g^{kl} (g_{il,j} +g_{jl,i} -g_{ij,l}), \hspace{1cm} i,j,k \in \{1,2,3\},
\end{equation*}
we have
\begin{equation} \label{eqn:symbols}
\left\{
    \begin{array}{ll}
    \Gamma_{11}^1 = \Gamma_{21}^2 = \Gamma_{12}^2 = -\Gamma_{22}^1 = \frac{2u_1}{1-r^2},\\
    \Gamma_{22}^2 = \Gamma_{12}^1 = \Gamma_{21}^1 = -\Gamma_{11}^2 = \frac{2u_2}{1-r^2},\\
    \Gamma_{ij}^k = 0 \hspace{5mm} \text{if } 3 \in \{i,j,k\}.
    \end{array}
\right.
\end{equation}
It is a well known fact that the canonical projections of a geodesic in the product Riemannian manifold are also geodesics. In particular, a projection of $\gamma$ onto the horizontal totally geodesic plane in $\hr$ is also a geodesic. Therefore $\gamma_2(t)=0$ because the only geodesics emanating from the origin in $\mathbb{H}^2$ are the rays. Putting (\ref{eqn:symbols}) in (\ref{eqn:geod eqn}), we get
\begin{eqnarray*}
\left\{
    \begin{array}{lll}
    \gamma_1''(t)+\Gamma_{11}^1 (\gamma_1'(t))^2 &=&0,\\
    \gamma_3''(t)&=&0.
    \end{array}
\right.
\end{eqnarray*}
With the given initial conditions, one may obtain as follows:
\begin{equation*}
\gamma(t)=\left(\tanh ct, 0,\sqrt{1-4c^2} t\right).
\end{equation*}

\begin{lem}
The Jacobi field along $\gamma$ with the initial condition $J(0)=(0,0,0)$ and $J'(0)=\omega(0)$ in $\hr$ is given by
\begin{equation} \label{lem:Jacobi}
J(t)=\left( t \omega_1(t), \frac{\sinh 2ct}{2c} \omega_2(t), t \omega_3(t) \right),
\end{equation}
where $\omega(t)=(\omega_1(t),\omega_2(t),\omega_3(t))$ is a parallel vector field along $\gamma$ with $\gamma'(t) \cdot \omega(t) = 0$ and $|\omega(t)| = 1$.
\end{lem}

\begin{proof}
We derive the Riemannian curvature tensor of $\hr$ using that of $\mathbb{H}^2$ as follows:
\begin{eqnarray*}
R_{ijk}^l=
\left\{
    \begin{array}{cl}
    \left(\frac{2}{1-r^2}\right)^2 \left(-\delta_{ik} \delta_{jl} +\delta_{jk} \delta_{il}\right), & i,j,k,l=1,2,\\
    0, & \text{otherwise}.
    \end{array}
\right.
\end{eqnarray*}
Since $\gamma_2(t)=0$, the Jacobi equation
\begin{equation*}
{J_l}''(t)+\sum_{i,j,k=1}^3 R_{ijk}^l {\gamma_i}'(t) {\gamma_k}'(t) J_j(t)=0, \hspace{5mm} l=1,2,3,
\end{equation*}
becomes
\begin{equation*}
\left\{
    \begin{array}{lll}
    {J_1}''(t)&=&0,\\
    {J_2}''(t)+R_{121}^2 \cdot ({\gamma_1}'(t))^2 J_2(t)&=&0,\\
    {J_3}''(t)&=&0.
    \end{array}
\right.
\end{equation*}
Along the geodesic $\gamma$, $R_{121}^2 \cdot \left({\gamma_1}'(t)\right)^2 =-\left(\frac{2}{1-r^2}\right)^2 \cdot \left(c(1-x^2)\right)^2= -4c^2$.
Solving the equation
\begin{equation*}
\left\{
    \begin{array}{lll}
    {J_1}''(t)&=&0,\\
    {J_2}''(t)-4c^2 J_2(t)&=&0,\\
    {J_3}''(t)&=&0
    \end{array}
\right.
\end{equation*}
with the given initial conditions $J(0)=(0,0,0)$ and $J'(0)=\omega(0)$, we can obtain (\ref{lem:Jacobi}).
\end{proof}

\begin{lem} {\rm (\cite{CG1}, Lemma 2)} \label{lem:laplacian}
Let $f$ be a smooth function on an $n$-dimensional Riemannian manifold $M$ and $\Sigma$ an $m$-dimensional submanifold of $M$. Let $\overline{\nabla}$ and $\overline{\triangle}$ be the connection and Laplacian on $M$ respectively, and $\triangle_\Sigma$ the Laplacian on $\Sigma$. If $H$ is the mean curvature vector of $\Sigma$ in $M$, then
\begin{equation} \label{eqn:laplacian}
\triangle_\Sigma f = (\overline{\triangle} f) | \Sigma + H f - \sum_{\alpha=m+1}^{n} \overline{\nabla}^2 f(\overline{e}_\alpha,\overline{e}_\alpha),
\end{equation}
where $H f$ is the directional derivative of $f$ in the direction of the mean curvature vector $H$ and $\overline{e}_{m+1},\cdots,\overline{e}_n$ are orthonormal vectors which are perpendicular to $\Sigma$.
\end{lem}

{\bf Proof of Proposition \ref{prop:hr subh}. }
Let $\gamma$ be a unit speed geodesic emanating from $0$ and $S_\rho(0)$ be a geodesic sphere of radius $\rho$ centered at $0$ in $\hr$. It is convenient to use the exponential coordinates $(\rho, \phi, \theta)$, where $\phi$ is the angle between $\gamma'(0)$ and the $z$-axis in $T_0\hr$ and $\theta$ is the angle around the $z$-axis. Note that $\cos \phi =\sqrt{1-4c^2}$ and hence $2c=\sin \phi$. In terms of $\phi$ one can rewrite $\gamma(t)$ and $J(t)$.

There are globally defined coordinate vector fields corresponding to these coordinates. Now define new vector fields $\{V_1,V_2,V_3\}$ to be parallel to the above-mentioned coordinate vector fields on a neighborhood of $q=\gamma(\rho) \in S_\rho(0)$ such that
\begin{displaymath}
V_1(q)=\gamma'(\rho)=\frac{\partial}{\partial \rho},\hspace{3mm} V_i \cdot V_j=0
\end{displaymath}
and
\begin{displaymath}
\left|V_1\right|=1,\hspace{3mm} \left|V_2\right|=\rho,\hspace{3mm} \left|V_3\right|=\frac{\sinh(\rho \sin \phi)}{\sin \phi}.
\end{displaymath}

By {Lemma \ref{lem:laplacian}} and the minimality of $\Sigma$, (\ref{eqn:laplacian}) yields
\begin{equation} \label{eqn:laplcian2}
\triangle_\Sigma \log \rho = \overline{\triangle} \log \rho - \overline{\nabla}^2 \log \rho(n,n) = \tr \overline{\nabla}^2 \log \rho - \overline{\nabla}^2 \log \rho(n,n),
\end{equation}
where $n$ is the unit normal vector field of $\Sigma$ in $\hr$.

Let $\{h_{ij}\}$ and $\{\widetilde{\Gamma}_{ij}^k\}$ be the metric and Christoffel symbols corresponding to the vector fields $\{V_i\}$. If $F$ is a smooth function on $\hr$, then the Hessian of $F$ satisfies
\begin{equation} \label{eqn:Hessian}
\overline{\nabla}^2 F (V_i,V_j)=\sum_{k=1}^3 \frac{1}{{\sqrt{h_{ii}} \sqrt{h_{jj}}}}(F_{ij} - \widetilde{\Gamma}_{ij}^k F_k ).
\end{equation}
Rewrite the metric $\{h_{ij}\}$ by the matrix form,
\begin{equation*}
    (h_{ij})=
    \left(
    \begin{array}{ccc}
    1 & 0 & 0 \\
    0 & \rho^2 & 0 \\
    0 & 0 & \frac{\sinh^2(\rho \sin \phi)}{\sin^2 \phi}\\
    \end{array}
    \right).
\end{equation*}
In case that $F=\log \rho$, two out of the three directional derivatives vanish. Substituting $\{h_{ij}\}$ into (\ref{eqn:Hessian}), we get
\begin{equation*}
    \overline{\nabla}^2 \log \rho (V_i,V_j)=
    \left\{
    \begin{array}{ll}
      \frac{1}{h_{11}} \left((\log \rho)_{11} - \widetilde{\Gamma}_{11}^1 (\log \rho)_1 \right), & i=j=1,\\
      \frac{1}{\sqrt{h_{ii}} \sqrt{h_{jj}}} \left(- \widetilde{\Gamma}_{ij}^1 (\log \rho)_1 \right), & {\text{otherwise}}.
    \end{array}
    \right.
\end{equation*}
Hence it is enough to compute the terms $\{\widetilde{\Gamma}_{ij}^1\}$. If $i \neq j$ then $\widetilde{\Gamma}_{ij}^1=0$. The only non-zero terms are
\begin{equation*}
\left\{
\begin{array}{lll}
    \widetilde{\Gamma}_{22}^1&=-\frac{1}{2} h_{22,1}&=- \rho,\\
    \widetilde{\Gamma}_{33}^1&=-\frac{1}{3} h_{33,1}&=- \frac{\sinh(\rho \sin \phi) \cosh(\rho \sin \phi)}{\sin \phi}.
\end{array}
\right.
\end{equation*}
Therefore
\begin{equation*}
\begin{array}{ll}
    \overline{\nabla}^2 \log \rho (V_1,V_1) &= \frac{1}{h_{11}} \left((\log \rho)_{11} - \widetilde{\Gamma}_{11}^1 (\log \rho)_1 \right) = -\frac{1}{\rho^2},\\
    \overline{\nabla}^2 \log \rho (V_2,V_2)&=-\frac{1}{h_{22}} \widetilde{\Gamma}_{22}^1 (\log \rho)_1 = \frac{1}{\rho^2},\\
    \overline{\nabla}^2 \log \rho (V_3,V_3)&=-\frac{1}{h_{33}} \widetilde{\Gamma}_{33}^1 (\log \rho)_1 = \frac{1}{\rho} \sin\phi \coth(\rho \sin \phi).
\end{array}
\end{equation*}
In conclusion, the Hessian of $\log \rho$ in $\hr$ is obtained
\begin{equation} \label{eqn:Hessian2}
    \overline{\nabla}^2 \log \rho=\frac{1}{\rho^2} \cdot
    \left(
      \begin{array}{ccc}
        -1 & 0 & 0 \\
        0 & 1 & 0 \\
        0 & 0 & \rho \sin\phi \coth(\rho \sin \phi) \\
      \end{array}
    \right).
\end{equation}
Putting $n=(n_1,n_2,n_3)$ and applying (\ref{eqn:laplcian2}) together with (\ref{eqn:Hessian2}),
\begin{eqnarray*}
    \rho^2 \triangle_\Sigma \log \rho &=& \rho^2 \cdot \left( \tr \overline{\nabla}^2 \log \rho - \overline{\nabla}^2 \log \rho(n,n) \right)\\
    &=& (1-n_3^2) \rho \sin \phi \coth(\rho \sin \phi)  - (n_2^2-n_1^2).
\end{eqnarray*}
Now we claim $\rho \sin \phi \coth(\rho \sin \phi)\geq 1$ . If we define $f(s)=s\coth s$, $0 \leq s \leq \rho$ then $f'(s)=\coth s - \frac{s}{\sinh^2 s}= \frac{1}{\sinh^ 2 s} (\sinh s \cdot \cosh s -s)$. Since $\frac{\md}{\md s}(\sinh s \cdot \cosh s -s)=\cosh {2s} -1 \geq 0$ and $\lim_{s\rightarrow0} (\sinh s \cdot \cosh s -s) =0$, $f$ is a monotonically increasing function. Since $\lim_{s\rightarrow0} f(s)= \lim_{s\rightarrow0} s \cdot \frac{e^s +e^{-s}}{e^s -e^{-s}} =1$, one can conclude $f(s)\geq1$.
Hence
\begin{equation*}
    \rho^2 \triangle_\Sigma \log \rho \geq 1+n_1^2 - (n_2^2+n_3^2) \geq 1+n_1^2 - |n|^2 = n_1^2 \geq 0. \qed
\end{equation*}

\begin{rem} \label{rem:hr subh}
{\rm If the equality holds in {Proposition \ref{prop:hr subh}} at a point $q \in \Sigma$ then either $n=(0,0,1)$ or both $n_1$ and $\sin \phi$ at $q$ vanish. In particular, if $\Sigma$ is a totally geodesic vertical plane containing $p$ and $q$, then $n=(0,0,1)$ at any $q \in \Sigma$ and thus $\triangle_{\Sigma} \log \rho \equiv 0$ on $\Sigma$, and vice versa. Note that $\sin \phi$ vanishes at $q$ if and only if $q$ lies in the vertical geodesic through $p$ in $\hr$.}
\end{rem}

\begin{df}
Let $\Pi_0$ be a totally geodesic vertical plane in $\hr$ such that $\Pi_0=\{(u_1,u_2,z) \in \hr |u_2=0\}$. A surface $\Sigma$ in $\hr$ is a {\it horizontal graph} of a function $f$ over $\Pi$ if
\begin{displaymath}
\Sigma=\left\{(u_1,u_2,z) \in \hr|\varphi^2=f\left(\varphi^1, \varphi^3\right)\right\},
\end{displaymath}
where $\varphi=(\varphi^1, \varphi^2, \varphi^3)$ is an isometry in $\hr$ such that $\varphi(\Pi)=\Pi_0$.
\end{df}

\begin{df} {\label{def:regular}}
Let $\Sigma$ be a complete minimal surface in $\hr$. $\Sigma$ is said to be {\it vertically regular at infinity} in $\hr$ if there is a compact subset $K \subset \hr$ such that
 \begin{enumerate}
   \item[1)] $\Sigma \sim K$ consists of $k$ components $\Sigma_1,\cdots,\Sigma_k$;
   \item[2)] each $\Sigma_i$ is the horizontal graph of a function $f_i$ over the exterior of a bounded region in some totally geodesic vertical plane $\Pi_i \simeq \mathbb{H} \times \mathbb{R}$;
   \item[3)] each $f_i$ has the following asymptotic behavior for $r$ large and $\alpha>0$:
   \begin{equation*}
    f_i\rightarrow 0, \hspace{3mm} \partial_{x_i} f_i=O \left(\frac{1}{r^{\alpha}} \right)\rightarrow 0 \hspace{3mm} \text{and} \hspace{3mm} \partial_{z_i} f_i=O \left(\frac{1}{\sinh^{1+\alpha} r} \right)\rightarrow 0,
   \end{equation*}
   as $r \rightarrow \infty$, where $x_i,z_i$ are the
   coordinates on $\Pi_i \simeq \mathbb{H} \times \mathbb{R}$
   and $r$ is the distance from $(0,0) \in \Pi_i$.
 \end{enumerate}
\end{df}

We call these $\Sigma_i$ the ends of $\Sigma$.

\begin{rem}
{\rm Schoen \cite{S} defines the following. A complete minimal surface $\Sigma \subset \mathbb{R}^3$ is said to be regular at infinity if there is a compact subset $K \subset \Sigma$ such that $\Sigma \sim K$ consists of $r$ components $\Sigma_1,\cdots,\Sigma_r$ such that each $\Sigma_i$ is the graph of a function $f_i$ with bounded slope over the exterior of a bounded region in some plane $\Pi_i$. Moreover, if $x_1$ and $x_2$ are coordinates in $\Pi_i$, we require the $f_i$ have the following asymptotic behavior for $r=|x|$ large:
\begin{equation*}
    f_i(x)=a\log r +b+\frac{c_1 x_1}{r^2} +\frac{c_2 x_2}{r^2}+O(r^{-2})
\end{equation*}
for constants $a$, $b$, $c_1$, $c_2$ depending on $i$.

As the distance $r$ from the origin, $(0,0) \in \Pi_i$, goes to infinity, $f_i$ is dominated by $\log r$. It comes from the profile curve of an end of a catenoid in $\mathbb{R}^3$. Note that $\log r \rightarrow \infty$ as $r \rightarrow \infty$, so $\Sigma$ goes apart from any plane parallel to $\Pi_i$. However, since $\frac{\log r}{r} \rightarrow 0$ as $r \rightarrow \infty$, $\Sigma$ tends to approach a plane. In other words the radial projection onto $\mathbb{S}^2$ of the intersection of $\Sigma$ and a geodesic sphere of radius $\rho$ converges uniformly as $\rho \rightarrow \infty$ to an equator, with multiplicities, of $\mathbb{S}^2$.\\

But in $\hr$ we have a different situation as follows:
\begin{itemize}
\item[] $\hr$ is not isotropic and homotheties are not isometries. So the behaviors of the components of a minimal surface outside a compact set in $\hr$ are different depending on whether they are vertical, horizontal, or mixed. In this section we deal with only the vertical cases.
\end{itemize}
}
\end{rem}

The following theorem is the main result of Section \ref{sect:hr}.

\begin{thm}
Let $\Sigma$ be a proper minimal surface in $\hr$. If $\Sigma$ is vertically regular at infinity in $\hr$ and has two ends, then $\Sigma$ is embedded.
\end{thm}

\begin{proof}
Let $p$ be a point in $\Sigma$ and $\rho(\cdot)=\dist(p,\cdot)$ in $\hr$. Since $\Sigma$ is proper in $\hr$, though $p$ is not in $K$, we can find new compact subset $\widetilde{K} \subset \hr$ satisfying all conditions in Definition \ref{def:regular}. For convenience, denote $\widetilde{K}$ by $K$, i.e. without loss of generality we may assume that $p \in K$. Since $\hr$ is homogeneous, we may also assume that $p=(0,0,0) \in \hr$.

By definition, each $\Sigma_i$ is the horizontal graph of the function $f_i$ over the exterior of a bounded region in $\Pi_i$. Let $x_i,z_i$ be the coordinates on $\Pi_i \simeq \mathbb{H} \times \mathbb{R}$ and let $p_i=(0,0)$ with respect to this coordinates. Then we can assume that $p_i \in \mathbb{H}^2 \times \{0\}$. Let $\varphi_i$ be the isometry in $\hr$ such that $\varphi_i(\Pi_i)=\Pi_0$, $\varphi_i(p_i)=p \in \hr$ and $\varphi_i$ preserves $\mathbb{R}$-axis. Let $S_r(p_i)$ be the geodesic circle of radius $r$ centered at $p_i$ on $\Pi_i$. Define $C_r^h(p_i)$ to be
\begin{equation*}
C_r^h(p_i):=\varphi_i^{-1}(\{(u_1,u_2,z) \in \hr|(u_1,0,z) \in \varphi_i(S_r(p_i)),u_2 \in [-h,h]\}).
\end{equation*}
Put $K_r^h=K \cup {C_r^h(p_1)} \cup {C_r^h(p_2)}$. The fact that $f_i\rightarrow 0$ as $r\rightarrow \infty$ implies that for sufficiently large $r$ there is $h>0$, which does not depend on $r$, such that $\Sigma \sim K_r^h$ consists of only two connected components of $\Sigma$. If necessary, we can enlarge $K$ since $\Sigma$ is proper. Furthermore we consider ${\Sigma \cap \partial C_r^h(p_i)}$ as a horizontal graph over $S_r(p_i)$.

Integrating $\triangle_\Sigma \log \rho$ in $\Sigma \cap K_r^h \sim B_\epsilon(p)$ and applying the divergence theorem gives
\begin{equation*}
    \int_{\Sigma \cap K_r^h \sim B_\epsilon(p)} \triangle_\Sigma \log \rho
    =\int_{\Sigma \cap \partial C_r^h(p_1)} \frac{1}{\rho} \frac{\partial \rho}{\partial \nu}
    +\int_{\Sigma \cap \partial C_r^h(p_2)} \frac{1}{\rho} \frac{\partial \rho}{\partial \nu}
    +\int_{\Sigma \cap \partial B_\epsilon(p)} \frac{1}{\rho} \frac{\partial \rho}{\partial \nu},
\end{equation*}
where $\nu$ is the outward unit conormal vector to $\Sigma \cap K_r^h \sim B_\epsilon(p)$.

Near $p$, $\sinh \rho \rightarrow \rho$ uniformly and $\Sigma$ is close to $T_p \Sigma$. Hence
\begin{equation*}
\vol(\Sigma \cap \partial B_\epsilon(p)) \rightarrow 2 \pi \epsilon \Theta_{\Sigma} (p)
\end{equation*}
and $\frac{\partial \rho}{\partial \nu} \rightarrow -1$ uniformly as $\epsilon \rightarrow 0$. So
\begin{equation*}
2 \pi \Theta_{\Sigma} (p) \leq -\int_{\Sigma \cap \partial B_\epsilon(p)} \frac{1}{\rho} \frac{\partial \rho}{\partial \nu}.
\end{equation*}
Therefore
\begin{equation} \label{eqn:hr0}
    2 \pi \Theta_{\Sigma}(p) \leq \int_{\Sigma \cap \partial C_r^h(p_1)} \frac{1}{\rho} \frac{\partial \rho}{\partial \nu}
    +\int_{\Sigma \cap \partial C_r^h(p_2)} \frac{1}{\rho} \frac{\partial \rho}{\partial \nu} - \int_{\Sigma \cap K_r^h \sim B_\epsilon(p)} \triangle_\Sigma \log \rho.
\end{equation}

Since $\dist(p_i,q) \leq \dist(p,q)$ for $q \in \Sigma_i$ sufficiently far from $p$, we have, $r \leq \rho$ and hence
\begin{eqnarray*}
    \int_{\Sigma \cap \partial C_r^h(p_i)} \frac{1}{\rho} \frac{\partial \rho}{\partial \nu} &\leq&
    \int_{\Sigma \cap \partial C_r^h(p_i)} \frac{1}{\rho}\\
    &\leq& \int_{\Sigma \cap \partial C_r^h(p_i)} \frac{1}{r} = \frac{1}{r} \length({\Sigma \cap \partial C_r^h(p_i)})\nonumber.
\end{eqnarray*}
Parameterize ${\Sigma \cap \partial C_r^h(p_i)}$ by $t$, $0\leq t \leq 2\pi$, as follows:
\begin{displaymath}
\left(x(t),f_i(x(t),z(t)),z(t)\right)=\left(\tanh \frac{1}{2} r \cos t, f_i(x(t),z(t)),r \sin t\right),
\end{displaymath}
where $x_i(t)$ and $z_i(t)$ is denoted by $x(t)$ and $z(t)$, respectively, for convenience. Then we can compute directly
\begin{eqnarray*}
&&\frac{1}{r} \length({\Sigma \cap \partial C_r^h(p_i)})\\&=&\frac{1}{r} \int_0^{2\pi} \sqrt{
\frac{x'(t)^2+\left(\partial_x f_i \cdot x'(t)+\partial_z f_i \cdot z'(t)\right)^2}{\left(\frac{1-{x^*}^2}{2}\right)^2} +z'(t)^2
}\\
&=&\int_0^{2\pi} \sqrt{
\left(\frac{\frac{1-{x}^2}{2} \cdot \sin t}{\frac{1-{x^*}^2}{2}}\right)^2+
\left(\frac{\frac{1-{x}^2}{2} \cdot \partial_x f_i \cdot (-\sin t)+\partial_z f_i \cdot \cos t}{\frac{1-{x^*}^2}{2}}
\right)^2+\cos^2 t
},
\end{eqnarray*}
where $x^*$ is determined by an orthogonal projection of the graph onto $\mathbb{H}^2$.
From the vertical regularity of $\Sigma$
\begin{eqnarray*}
    \left|\partial_z f_i\right|\cdot \left(\frac{1-{x^*}^2}{2}\right)^{-1}
    &=& O \left(\frac{1}{\sinh^{1+\alpha} r} \right)\cdot \left( \frac{{x^*}}{\sinh r^*} \right)^{-1}\\
    &=& O \left(\frac{1}{\sinh^{1+\alpha} r} \right)\cdot \left( 1+{\cosh r^*} \right) \longrightarrow 0,\\
    \left|\partial_x f_i\right|\cdot \left(\frac{1-{x}^2}{1-{x^*}^2}\right)
    &=& O \left(\frac{1}{r^{\alpha}} \right)\cdot \left(\frac{1+{\cosh r^*}}{1+{\cosh r}}\right) \longrightarrow 0,
\end{eqnarray*}
as $r\rightarrow \infty$, where $r^*=\log \frac{1+x^*}{1-x^*}$ the distance from $p_i$ in $\mathbb{H}^2$.

It implies that not only ${\Sigma \cap \partial C_r^h(p_i)}$ converges to $S_r(p_i)$ as a set but also the tangent vectors of ${\Sigma \cap \partial C_r^h(p_i)}$ converge to those of $S_r(p_i)$ uniformly. Therefore
\begin{equation} \label{eqn:hr2}
     \frac{1}{r} \length({\Sigma \cap \partial C_r^h(p_i)})\rightarrow \frac{1}{r} \length(S_r(p_i)) =2\pi,
\end{equation}
as $r\rightarrow \infty$. Note that the last equality holds because $\Pi_i \simeq \mathbb{H}\times \mathbb{R}$ is isometric to $\mathbb{R}^2$. Then (\ref{eqn:hr2}) implies that for every $\delta>0$, there exists $R$ such that
\begin{equation*}
    \left| \int_{\Sigma \cap \partial C_r^h(p_1)} \frac{1}{\rho} \frac{\partial \rho}{\partial \nu}
    +\int_{\Sigma \cap \partial C_r^h(p_2)} \frac{1}{\rho} \frac{\partial \rho}{\partial \nu} -4\pi \right| < \delta \hspace{3mm} \text{if } r>R.
\end{equation*}
{Remark \ref{rem:hr subh}} implies that
\begin{displaymath}
\int_{\Sigma \cap K_r^h \sim B_\epsilon(p)} \triangle_\Sigma \log \rho
 \end{displaymath}
is strictly positive since $\Sigma$ can not be a union of two totally geodesic vertical planes. Then there is $\delta_0 >0$ which does not depend on $r$ such that
\begin{equation*}
    \int_{\Sigma \cap K_r^h \sim B_\epsilon(p)} \triangle_\Sigma \log \rho > \delta_0.
\end{equation*}
If we take $\delta< \frac{\delta_0}{2}$ then (\ref{eqn:hr0}) deduces
\begin{equation*}
    2 \pi \Theta_{\Sigma}(p) < 4\pi +\delta  -\delta_0 < 4\pi +\frac{\delta_0}{2} -\delta_0 < 4\pi.
\end{equation*}
This completes the proof.
\end{proof}

\noindent Sung-Hong Min\\
Korea Institute for Advanced Study, 207-43 Cheongnyangni 2-Dong, Dongdaemun-Gu, Seoul 130-722, Korea\\
{\tt e-mail:shmin@kias.re.kr}
\end{document}